\documentclass[11pt,a4paper]{article}
\usepackage{amsmath,amssymb,amsthm,amscd,mathrsfs}
\usepackage{indentfirst}
\usepackage[top=25mm, bottom=30mm, left=30mm, right=30mm]{geometry}
\newtheorem{Def}{\bf Definition}[section]
\newtheorem{Thm}[Def]{\bf Theorem}
\newtheorem{Lem}[Def]{\bf Lemma}

\newtheorem{Pro}[Def]{\bf Proposition}
\newtheorem{Rem}[Def]{\bf Remark}

\newtheorem*{Que}{\bf Question}

\newtheorem*{claim1}{\bf Claim}

\newtheorem{case1}{\bf Case}
\newtheorem{ThmA}{\bf Theorem}
\newtheorem{ProA}[ThmA]{\bf Proposition}

\newtheorem{CorA}[ThmA]{\bf Corollary}

\newcommand{\R}{\mathbb{R}}
\newcommand{\C}{\mathbb{C}}
\newcommand{\Z}{\mathbb{Z}}
\newcommand{\F}{\mathbb{F}}

\newcommand{\N}{\mathbb{N}}
\newcommand{\B}{\mathbb{B}}
\newcommand{\K}{\mathbb{K}}
\newcommand{\M}{\mathbb{M}}
\newcommand{\Ad}{\operatorname{Ad}}
\newcommand{\id}{\text{\rm id}}
\newcommand{\Aut}{\operatorname{Aut}}
\newcommand{\Tr}{\mathord{\text{\rm Tr}}}
\newcommand{\ovt}{\mathbin{\overline{\otimes}}}

\title{\bf On fundamental groups of tensor product \\$\rm\bf II_1$ factors}
\author{Yusuke Isono\thanks{Research Institute for Mathematical Sciences, Kyoto University, 606-8502, Kyoto, Japan \protect \\ \hspace{1.47em} E-mail: \texttt{isono@kurims.kyoto-u.ac.jp}}}
\date{}
\begin{document}
\maketitle

\begin{abstract}
	Let $M$ be a $\rm II_1$ factor and let $\mathcal{F}(M)$ denote the fundamental group of $M$. In this article, we study the following property of $M$: for any $\rm II_1$ factor $B$, we have $\mathcal{F}(M\ovt  B)=\mathcal{F}(M)\mathcal{F}(B)$. We prove that for any subgroup $G\leq \mathbb{R}^*_+$ which is realized as a fundamental group of a $\rm II_1$ factor, there exists a $\rm II_1$ factor $M$ which satisfies this property and whose fundamental group is $G$. Using this, we deduce that if $G,H \leq \R^*_+$ are realized as fundamental groups of $\rm II_1$ factors, then so are groups $G\cdot H$ and $G\cap H$.
\end{abstract}

\section{\bf Introduction and main theorems}

	In their pioneering work, Murray and von Neumann introduced the fundamental group as an invariant of $\rm II_1$ factors \cite{MV43}. For a $\rm II_1$ factor $M$ with trace $\tau$, the \textit{fundamental group} is defined as
\begin{equation*}
	\mathcal{F}(M):= \left\{ \frac{\tau(p)}{\tau(q)}\in \mathbb{R}^*_+ \, \middle| \, p,q \text{ are projections in $M$ with } pMp\simeq qMq \right\}.
\end{equation*}
Murray and von Neumann proved the hyperfinite (or \textit{amenable}) $\rm II_1$ factor has the full fundamental group $\mathbb{R}^*_+$, and then asked the general behavior of this invariant. 
Indeed, the fundamental group is the most well known invariant for $\rm II_1$ factors, and to determine which subgroup of $\R^*_+$ appears as a fundamental group is a long-standing open problem in the von Neumann algebra theory. 

	Computation of fundamental groups, however, is a hard problem. Indeed, $\rm II_1$ factors $pMp$ and $qMq$ share a lot of properties in common, so it is very difficult to distinguish them. Thus very few computations have been done until recently. 
Connes proved that $L\Gamma$, where $\Gamma$ is an ICC property (T) group, has a countable fundamental group \cite{Co80}, which is the first example of a $\rm II_1$ factor with fundamental group not equal to $\R^*_+$. Voiculescu and R$\rm \breve{a}$dulescu proved $\mathcal{F}(L\F_\infty)$ has the full fundamental group $\R^*_+$ \cite{Vo89,Ra91}.

	In 2001, Popa introduced a new framework to study this problem \cite{Po01}. He developed a way of identifying Cartan subalgebras and then reduced this computation problem for a certain class of $\rm II_1$ factors to the one for corresponding orbit equivalence relations. 
Thus combined with Gaboriau's work on orbit equivalence relations \cite{Ga99,Ga01}, Popa obtained the first example of a $\rm II_1$ factor which has the trivial fundamental group. 

	Much progress has been made by this new technology in the von Neumann algebra theory. The study in this new framework is now called the \textit{deformation/rigidity theory}. 
Thus, a lot of computations of fundamental groups have been done in the last decade. 

	We say that a subgroup $G \leq \mathbb{R}_+^*$ is in the \textit{class} $\mathcal{S}_{\rm factor}$ if there is a $\rm II_1$ factor $M$ with separable predual such that $\mathcal{F}(M)=G$. 
Popa proved that any countable subgroup of $\R_+^*$ is contained in $\mathcal{S}_{\rm factor}$ \cite{Po03}. Popa and Vaes proved that $\mathcal{S}_{\rm factor}$ contains many uncountable subgroups in $\R_+^*$ \cite{PV08a}. See \cite{Po04,IPP05,Po06a,Ho07,PV08b,De10} for other calculations of fundamental groups. We note that, while this new theory provides a lot of concrete examples, very few general properties for the class $\mathcal{S}_{\rm factor}$ are known so far. (See Proposition \ref{general} below.)

	The aim of this article is to study fundamental groups of tensor product $\rm II_1$ factors. For this, recall that for a $\rm II_1$ factor $M$ and $t>0$, the \textit{amplification} $M^t$ is defined (up to $*$-isomorphism) as $pMp \ovt \mathbb{M}_n$ for any $n\in\mathbb{N}$ with $t\leq n$ and any projection $p\in M$ with trace $t/n$. 
It is then easy to verify that 
\begin{itemize}
	\item $\mathcal{F}(M)=\{t\in \mathbb{R}^*_+ \mid M\simeq M^t \}$;
	\item $(M_1\ovt  M_2)^{st} \simeq M_1^s\ovt  M_2^{t}$ for $\rm II_1$ factors $M_i$ and $s,t>0$.
\end{itemize}
They particularly imply the following inclusion:
	$$\mathcal{F}(M_1\ovt  M_2)\supset\mathcal{F}(M_1)\mathcal{F}(M_2)$$ 
for any $\rm II_1$ factors $M_1$ and $M_2$. 
It is likely that the converse inclusion also holds true at the first glance. However, as we emphasized, the computation of fundamental groups is a hard problem, and so we know very little about this converse inclusion until recently. 

	To study this inclusion is actually the main purpose of this article. 
This is a natural question, since it provides a quite useful formula for fundamental groups of tensor product II$_1$ factors. Indeed, if the converse inclusion holds, then one actually has an equation, so the computation of fundamental groups for the tensor product can be reduced to the one for each tensor component. 
Here we state this problem in the following precise form. 
\begin{Que}
	For which $\rm II_1$ factors $M_1$ and $M_2$, do we obtain the equation		$$\mathcal{F}(M_1\ovt  M_2) = \mathcal{F}(M_1)\mathcal{F}(M_2)?$$
\end{Que}
\noindent
We do not believe that all II$_1$ factors satisfy it, although, to the best of our knowledge, any example of II$_1$ factors which do \textit{not} satisfy this equation is not known.


	In the deformation/rigidity theory, Ozawa and Popa provided the first class of $\rm II_1$ factors that satisfy this equality. They proved that if each $M_i$ is a free group factor, then the tensor product satisfies a \textit{unique prime factorization} result and particularly the equation above holds true \cite{OP03}. See \cite{Pe06,Sa09,CSU11,SW11,Is14,CKP14,HI15,Ho15} for other classes of factors which satisfy the unique prime factorization result.

	In this article, we further develop Ozawa--Popa's strategy. We particularly study the following property for a II$_1$ factor $M$:
	$$\mathcal{F}(M\ovt  B)=\mathcal{F}(M)\mathcal{F}(B) \text{ for \textit{any} $\rm II_1$ factor } B.$$
Obviously this condition is stronger than the one Ozawa-Popa obtained for free groups factors, since the factor $B$ in the condition can be arbitrary. 
If this condition holds, we say that $M$ satisfies the \textit{tensor factorization property for fundamental groups} (say, \textit{property (TFF)} in short). 


	Our first theorem provides examples of $\rm II_1$ factors having property (TFF). See \cite[Definitions 12.3.1 and 15.1.2]{BO08} for definitions of weak amenability and bi-exactness (and note that free groups, more generally hyperbolic groups, satisfy them).

\begin{ThmA}\label{thmA}
	Let $M$ be one of the following $\rm II_1$ factors.
	\begin{itemize}
		\item A group $\rm II_1$ factor $L\Gamma$, where $\Gamma$ is an ICC, non-amenable, weakly amenable, and bi-exact group.
		\item A free product $\rm II_1$ factor $M_1*M_2$, where $M_1$ and $M_2$ are diffuse (and tracial). 
		\item A group $\rm II_1$ factor $L(\Delta\wr\Lambda)$, where $\Delta$ is a non-trivial amenable group and $\Lambda$ is a non-amenable group. 
	\end{itemize}
	Then $M$ satisfies the property (TFF).
\end{ThmA}

As a corollary of this theorem, we obtain the main observation of this article. In fact, the following corollary states \textit{general} properties for the class $\mathcal{S}_{\rm factor}$. Although it is not enough to answer the aforementioned question by Murray and von Neumann, this is an interesting consequence since there are very few general properties for the class $\mathcal{S}_{\rm factor}$ as we mentioned.

\begin{CorA}\label{corB}
	For any $G \in \mathcal{S}_{\rm factor}$, there is a $\rm II_1$ factor $M$ with separable predual and with the property (TFF) such that $\mathcal{F}(M)=G$. 

The class $\mathcal{S}_{\rm factor}$ admits the following properties.
	\begin{itemize}
		\item Stability under multiplication: for any $G,H\in \mathcal{S}_{\rm factor}$, the group $G\cdot H$ is  in $\mathcal{S}_{\rm factor}$.
		\item Stability under countable intersection: for any $G_n\in \mathcal{S}_{\rm factor}$ (possibly $G_n=G_m$ for $n\neq m$), $n\in \N$, the group $\bigcap_n G_n$ is in $\mathcal{S}_{\rm factor}$.
	\end{itemize}
\end{CorA}

	We note that the proof of the first statement in this corollary in fact shows the following: if we put 
$N:=L\F_n * L(\Z^2\rtimes\mathrm{SL}(2,\Z))$, then for \textit{any} $\rm II_1$ factor $B$ we have 
	$$\mathcal{F}(B) = \mathcal{F}(*_{\N} (B\ovt N)).$$
Thus combined with Theorem \ref{thmA}, the free product $\rm II_1$ factor $*_{\N} (B\ovt N)$ does the work. 

	The proof of Theorem \ref{thmA} uses the idea in our previous paper \cite{Is14}, in which we introduced another notion of primeness for II$_1$ factors. Recall that a $\rm II_1$ factor $M$ is said to be \textit{prime} if it does not have a tensor decomposition as $\rm II_1$ factors, namely, if it has a decomposition $M=M_1\ovt  M_2$, then at least one $M_i$ must be of type I. 
Obviously this definition comes from the notion of prime numbers in the number theory. 

	Actually there are two equivalent notions of prime numbers. 
Recall that a number $p\in\mathbb{N}$ is \textit{irreducible} if for any $q,r\in\mathbb{N}$ with $p=qr$, we have $q=1$ or $r=1$; and is \textit{prime} if for any $q,r,s\in\mathbb{N}$ with $pq=rs$, we have $p\mid r$ or $p\mid s$. 
In the von Neumann algebra theory, we adapt the first one (i.e.\ irreducibility) as a definition of primeness. In \cite[Section 5]{Is14}, we introduced a different notion of primeness, which corresponds to the second one as follows. To distinguish two primeness, we name it \textit{strongly prime}. 
\begin{itemize}
	\item We say a $\rm II_1$ factor $M$ is \textit{strongly prime} if for any $\rm II_1$ factors $B,K$ and $L$ with $M\ovt  B=K\ovt  L$, there is a unitary $u\in\mathcal{U}(M\ovt  B)$ and $t>0$ such that, under the identification $K\ovt L = K^t \ovt L^{1/t}$,
we have $uM u^*\subset K^t$ or $uMu^*\subset L^{1/t}$.
\end{itemize}
Here we identify each tensor component as a subalgbera (e.g.\ $M = M \otimes \C \subset M\ovt B$).

	Our second main theorem treats examples of strongly prime factors. Note that the first item in this theorem was already obtained in our previous article \cite[Theorem 5.1]{Is14}. We also note that the first and the second item in the theorem treat exactly the same ones as in Theorem \ref{thmA}.

\begin{ThmA}\label{thmC}
	Let $M$ be one of the following $\rm II_1$ factors.
	\begin{itemize}
		\item A group $\rm II_1$ factor $L\Gamma$, where $\Gamma$ is a non-amenable, ICC, weakly amenable, and bi-exact group.
		\item A free product $\rm II_1$ factor $M_1*M_2$, where $M_1$ and $M_2$ are diffuse. 
		\item A group $\rm II_1$ factor $L(\Delta\wr\Lambda)$, where $\Delta$ is a non-trivial amenable group and $\Lambda=\Lambda_1\times \Lambda_2$ is a direct product of any group (possibly trivial) $\Lambda_1$ and a non-amenable, weakly amenable, and bi-exact group $\Lambda_2$. 
	\end{itemize}
Then $M$ is strongly prime.
\end{ThmA}

	In section \ref{some observations}, we will show that the property (TFF) has a sufficient condition similar to strong primeness (Lemma \ref{tensor formula}), and hence strong primeness is actually a sufficient condition to the property (TFF) (Proposition \ref{s-prime to TF}). We note that strong primeness implies primeness, but the converse fails (Propositions \ref{s-prime to TF} and \ref{prime not to s-prime}). 

	We will also discuss unique prime factorization result, using the strong primeness. This particularly provides the first example of unique prime factorization result for \textit{infinite} tensor products. Below we say that a II$_1$ factor $M$ is \textit{semiprime} if for any tensor decomposition $M=M_1\ovt M_2$, at least one $M_i$ is amenable. 
The reason we use semiprimeness is that any infinite tensor product factor $M$ is McDuff (i.e.\ $M\simeq M\ovt R$ for the hyperfinite II$_1$ factor $R$), so tensor components are determined up to tensor product with $R$. 

\begin{ProA}\label{prime factorization}
	Let $m,n\in \N \cup \{\infty\}$. Let $M_i$ be strongly prime $\rm II_1$ factors, and $N_j$ any $\rm II_1$ factors such that $\ovt_{i=1}^m M_i = \ovt_{j=1}^n N_j =:M$. 
Then there is a unique map $\sigma\colon \{1,2,\ldots, m\} \to \{1,2,\ldots, n\}$ such that $M_i \preceq_M N_{\sigma(i)}$ for all $i\in \{1,2,\ldots, m\}$. 
In this case, the following statements hold true.
\begin{itemize}
	\item The map $\sigma$ is surjective if and only if all $N_j$ are non-amenable.
	\item The map $\sigma$ is injective if and only if all $N_{\sigma(i)}$ are semiprime. 
\end{itemize}
	Thus the map $\sigma$ is bijective if all $N_j$ are non-amenable and semiprime. In this case for each $i\in \{1,2,\ldots, m\}$, $N_{\sigma(i)}$ is isomorphic to $M_i^{t_i}\ovt P_i$ for some $t_i>0$ and some amenable factor $P_i$.
\end{ProA}

	In the proposition, if we assume all $N_j$ are prime, then the map $\sigma$ is bijective and $M_i$ and $N_{\sigma(i)}$ are stably isomorphic for all $i$. 
	We note that the map $\sigma$ in the proposition is surjective whenever $m<\infty$, since $M=M_1\ovt\cdots \ovt M_m$ is full (Proposition \ref{prime not to s-prime} and \cite[Corollary 2.3]{Co75}) and so $N_j$ can not be amenable.

\bigskip

\noindent
{\bf Acknowledgement.} The author would like to thank R$\rm \acute{e}$mi Boutonnet, Cyril Houdayer, Adrian Ioana, Narutaka Ozawa and Stefaan Vaes for fruitful conversations. He also thank the referee for pointing out that a part of the proof of Theorem \ref{thmC} can be simplified using Lemma \ref{Sa09}. 
He was supported by JSPS, Research Fellow of the Japan Society for the Promotion of Science.

\section{\bf Preliminaries}\label{Preliminaries}

In this article, all von Neumann algebras that we consider are assumed to be finite and $\sigma$-finite, namely, they admit faithful normal tracial states.

\subsection*{\bf General properties for the class $\mathcal{S}_{\rm factor}$}

	Let $M$ be a II$_1$ factor and $\Tr$ a trace on $M\ovt \B(\ell^2)$. Then since the trace on $M\ovt \B(\ell^2)$ is unique up to scalars, there is a  homomorphism 
	$$\mathrm{Mod}\colon \Aut(M\ovt \B(\ell^2))\to \mathcal{F}(M); \quad \Tr\circ\alpha=\mathrm{Mod}(\alpha)\Tr, \quad \alpha\in \Aut(M).$$
It is then not difficult to see that $\mathrm{Mod}$ is surjective and continuous (with respect to the \textit{u-topology} on $\Aut(M\ovt \B(\ell^2))$). 
Since $\Aut(M\ovt \B(\ell^2))$ with the u-topology is a Polish group when $M$ has separable predual, we get the following proposition. This is the only known general property for the class $\mathcal{S}_{\rm factor}$ so far.

\begin{Pro}\label{general}
	For any group $G\in \mathcal{S}_{\rm factor}$, there is a Polish group $P$ and a continuous surjective homomorphism from $P$ onto $G$.
\end{Pro}

	 Using this, one can show for example that any group in $\mathcal{S}_{\rm factor}$ is a Borel subset of $\R^*_+$. 
Our main observation will provide the second general property for $\mathcal{S}_{\rm factor}$.

\subsection*{\bf Popa's  intertwining technique}\label{intertwining technique}

	We recall Popa's intertwining theorem. This is the main tool in the deformation/rigidity theory.

\begin{Thm}[{\cite{Po01,Po03}}]\label{Popa embed}
	Let $M$ be a finite von Neumann algebra with trace $\tau$, $p$ and $q$ projections in $M$, $A\subset pMp$ and $B\subset qMq$ von Neumann subalgebras with $\tau$-preserving conditional expectations $E_A$ and $E_B$. Then the following conditions are equivalent.
	\begin{itemize}
		\item[$\rm (i)$] There exist non-zero projections $e\in A$, $f\in B$, a unital normal $*$-homomorphism $\theta \colon eAe \rightarrow fBf$, and a partial isometry $v\in eMf$ such that $v\theta(x)=xv$ for all $x\in eAe$.
		\item[$\rm (i)'$] There exist a nonzero normal $*$-homomorphism $\psi\colon A \rightarrow B\ovt \mathbb{M}_n$ for some $n\in \mathbb{N}$ and a nonzero partial isometry $w\in (p\otimes e_{1,1})(M\ovt \mathbb{M}_{n})$ such that $w\psi(x)=(x\otimes e_{1,1})w$ for all $x\in A$, where $(e_{i,j})_{i,j}$ is a fixed matrix unit in $\mathbb{M}_n$. 
		\item[$\rm (ii)$] There exists no net $(w_i)_i$ of unitaries in $A$ such that $\|E_{B}(b^*w_i a)\|_{2,\tau} \rightarrow 0$ for any $a,b\in pMq$.
		\item[$\rm (iii)$] There exists a positive element $d\in p\langle M,B\rangle p\cap A'$ such that $0<\mathrm{Tr}_{\langle M,B\rangle}(d)<\infty$, where $\mathrm{Tr}_{\langle M,B\rangle}$ is the canonical trace on $\langle M,B\rangle$ (with respect to $\tau$).
	\end{itemize}
	We write $A\preceq_M B$ if one of these conditions holds.
\end{Thm}

	Note that when $B = \C$, $A\not\preceq_M\C$ if and only if $A$ is diffuse. 
We next observe some elementary lemmas.

\begin{Lem}[{\cite[Lemma 4.6]{HI15}}]\label{tensor case}
	Let $M$ and $N$ be finite von Neumann algebras, $p$ a projection in $M$, $A \subset  M $, $N_0\subset N$ and $B\subset M$ finite von Neumann subalgebras. 
Then $A \preceq_M B$ if and only if $A \otimes \C 1_N \preceq_{M \ovt N} B \otimes \C 1_N$ if and only if $A \otimes N_0 \preceq_{M \ovt N} B \ovt N$.
\end{Lem}

\begin{Lem}\label{crossed product case}
	Let $B$ be a finite von Neumann algebra and $\Gamma$ a discrete group acting on $B$ as a trace preserving action. Write $M:=B\rtimes\Gamma$. Then $M\preceq_M B$ if and only if $L\Gamma\preceq_M B$ if and only if $\Gamma$ is a finite group. 
\end{Lem}
\begin{proof}
	If $\Gamma$ is a finite group, then the canonical trace of the basic construction $\langle M,B \rangle$ is finite. So by Theorem \ref{Popa embed}(iii), we get $M\preceq_M B$. 
If $\Gamma$ is infinite, then we can find a sequence $g_n\in \Gamma$ such that all $g_n$ are distinct with each other. Then it satisfies Theorem \ref{Popa embed}(ii) and hence $L\Gamma\not\preceq_MB$. Finally by Theorem \ref{Popa embed}$\rm (i)'$, it is obvious that $M\preceq_MB$ implies $L\Gamma\preceq_MB$.
\end{proof}

\begin{Lem}\label{free product case}
	Let $M=M_1*M_2$ be a tracial free product von Neumann algebra, $p\in M_1$ a projection, and let $A\subset pM_1p$ be a diffuse von Neumann subalgebra. Then we have $A\not\preceq_MM_2$.
\end{Lem}
\begin{proof}
	We may assume $M_2\neq\C$. Let $(u_n)_n$ be a sequence of unitaries in $A$ which converges to 0 weakly. By simple calculations, one can show that if each $a,b\in M$ is a scalar or a reduced word, then $\|E_{M_2}(b^*u_na)\|_2$ converges to 0 as $n \rightarrow \infty$. Hence by Theorem \ref{Popa embed} (ii),  $A\not\preceq_MM_2$ holds.
\end{proof}

	In the lemma below, we denote the \textit{normalizer} for an inclusion $B \subset M$ by $\mathcal{N}_M(B):= \{ u \in \mathcal{U}(M) \mid uBu^* = B \}$.

\begin{Lem}\label{Sa09}
	Let $B\subset M$ be finite von Neumann algebras, $p$ a projection in $M$, and $A, P\subset pM p$ von Neumann subalgebras. Assume that $A$ and $P$ commute. Assume $A\preceq_M B$ and $P\preceq_MB$. If $B$ is regular (i.e.\ $\mathcal{N}_M(B)''=M$) and $\mathcal{N}_{pMp}(A)'\cap pMp= \mathbb{C}p$, then we have $(A\vee P)\preceq_M B$.
\end{Lem}
\begin{proof}
	We follow the proof of \cite[\textrm{Lemma 33}]{Sa09}. 

	By Theorem \ref{Popa embed} (iii), we find a positive element $d_A\in A' \cap p\langle M, B \rangle p$ with $0<\Tr_{\langle M, B \rangle}(d)<\infty$. Taking a spectral projection, we may assume $d_A$ is a projection. 
Observe that for any $u\in \mathcal{N}_{pMp}(A)$ and $v\in \mathcal{N}_M(B)$, the element $v^{\rm op} u d_A u^*(v^{\rm op})^*$ satisfies the same condition as the one on $d_A$, where $v^{\rm op}$ is the right action of $v$ on $L^2(M)$ (we indeed have $\Tr_{\langle M, B \rangle}\circ \Ad v^{\rm op} = \Tr_{\langle M, B \rangle}$, since $(v^{\rm op})^* v e_B v^{\rm op} v^*=e_B$. See \cite[Exercise F.6]{BO08} for the construction of $\Tr_{\langle M, B \rangle}$ and use the fact that $\Tr_{\langle M, B \rangle}$ is uniquely determined by $\Tr_{\langle M, B \rangle}(x^*e_Bx) = \tau(x^*x)$ for $x\in M$). So the element $d:=\sup\{ v^{\rm op} u d_A u^*(v^{\rm op})^*\mid u\in \mathcal{N}_{pMp}(A),v\in \mathcal{N}_M(B)\}$ is contained in 
\begin{eqnarray*}
	&& A'\cap \mathcal{N}_{pMp}(A)' \cap p\langle M, B \rangle p \cap p(\mathcal{N}_M(B)^{\rm op})'p\\
	&=& \mathcal{N}_{pMp}(A)' \cap p\langle M, B \rangle p \cap pMp\\
	&=& \mathcal{N}_{pMp}(A)' \cap pMp = \mathbb{C}p.
\end{eqnarray*}
Hence we get $d=p$. Let now $d_P$ be a non-zero trace finite projection in $P' \cap p\langle M, B \rangle p$. Then since $d=p$, there are finite subsets $\mathcal{E} \subset \mathcal{N}_{pMp}(A)$ and $\mathcal{F}\subset \mathcal{N}_M(B)$ satisfying that $\vee_{u\in \mathcal{E},v\in\mathcal{F}}v^{\rm op} u d_A u^*(v^{\rm op})^*$ is not orthogonal to $d_P$. 
Thus up to exchanging $d_A$ with this element, we can assume $d_Ad_P\neq0$. 

	Consider a convex subset $K:=\overline{\mathrm{co}}^{\rm w}\{u d_A u^*\mid u\in \mathcal{U}(P) \} \subset \langle M,B \rangle$ and observe that $K$ is regarded as a subset in $L^2(\langle M, B\rangle, \Tr_{\langle M, B\rangle})$ which is $L^2$-norm bounded (e.g.\ \cite[Exercise F.3]{BO08}). Take the unique minimal $L^2$-norm element $\widetilde{d}$ in $K$. 
We have $u\widetilde{d}u^* = \widetilde{d}$ for any $u\in \mathcal{U}(P)$ by the uniqueness, and hence $\widetilde{d}$ is contained in $P'\cap A' \cap p\langle M, B \rangle p=(A\vee P)'\cap p\langle M, B \rangle p$. Observe that $\widetilde{d}$ is trace finite in $\langle M, B\rangle$ since so is $d_A$ (and $\Tr_{\langle M, B\rangle}$ is normal). Finally $\widetilde{d}$ is non-zero since for any $u\in \mathcal{U}(P)$,
\begin{eqnarray*}
	\langle ud_Au^* , d_P\rangle
	=\mathrm{Tr}_{\langle M, B\rangle}(ud_Au^*d_P)
	=\mathrm{Tr}_{\langle M, B\rangle}(d_Ad_P)>0,
\end{eqnarray*}
and so any $a \in K$ satisfies $\langle a , d_P \rangle =\mathrm{Tr}_{\langle M, B\rangle}(d_Ad_P)>0$. 
Thus we obtain $(A\vee P)\preceq_MB$. 
\end{proof}

\subsection*{\bf Relative amenability}\label{Relative amenability in general von Neumann algebras}

	We next recall relative amenability introduced in \cite{OP07}.

\begin{Def}[{\cite[Definition 2.2]{OP07}}]\upshape
	Let $M$ be a finite von Neumann algebra with trace $\tau$. Let $p\in M$ be a projection and $A\subset pMp$ and $B\subset M$ von Neumann subalgebras. 
We say \textit{$A$ is amenable relative to $B$ in $M$}, and write as $A \lessdot_MB$, if there exists a conditional expectation from $p\langle M,B\rangle p$ onto $A$ which restricts to a $\tau$-preserving expectation on $pMp$.
\end{Def}

\begin{Pro}[{\cite[Proposition 2.4(3)]{OP07}}]\label{relative amenable transitivity}
		Let $B\subset M$ and $A\subset pMp$ as above, and let $N\subset M$ be another von Neumann subalgebra. If $A \lessdot_MB$ and $B \lessdot_MN$, then $A \lessdot_MN$.
\end{Pro}

We record the following elementary lemma.

\begin{Lem}\label{relative amenable and amenable}
	Let $M$ and $B$ be finite von Neumann algebras. Then $M\ovt B \lessdot_{M\ovt B} B$ if and only if $M$ is amenable.
\end{Lem}

\section{\bf Some observations on tensor product $\bf\rm II_1$ factors}\label{some observations}

	In this section, we briefly review fundamental properties of tensor product II$_1$ factors to study the property (TFF) and strong primeness. 
	We say $M_1\ovt  \cdots \ovt  M_m=N_1\ovt  \cdots \ovt  N_n$ is a \textit{tensor decomposition as $\rm II_1$ factors} if each $M_i$ and $N_j$ is a $\rm II_1$ factor. 

\subsection*{\bf Property (TFF)}

	We first recall the following observation of Ozawa and Popa. This shows that, to see a unitary embedding on tensor products, we have only to find Popa's conjugacy ``$\preceq$'' introduced in Theorem \ref{Popa embed}. This allows us to reformulate strong primeness (Lemma \ref{lemma}), so that we can make use of results in the deformation/rigidity theory. 

\begin{Lem}[{\cite[\textrm{Proposition 12}]{OP03}}]\label{intertwining tensor}
	Let $M_1\ovt  M_2= N_1\ovt  N_2$ $(=:M)$ be a tensor decomposition as $\rm II_1$ factors. Then $N_1\preceq_M M_1$ if and only if 
there is a unitary element $u\in M$ and a decomposition $M= M_1^t\ovt  M_2^{1/t}$ for some $t>0$  such that $uN_1u^*\subset M_1^t$. 
\end{Lem}
\begin{Rem}\upshape
	In this lemma, we are having an identification $M= M_1^t\ovt  M_2^{1/t}$, using a non-canonical isomorphism $M_1\ovt M_2 \simeq M_1^t\ovt  M_2^{1/t}$. Since this isomorphism is given at the level of a partial isometry conjugacy of $M_1\ovt M_2 \ovt \M_n$ (for some large $n\in \N$), one can show that $N_1\preceq_M M_1$ if and only if $N_1\preceq_M M_1^t$ for any $t>0$ and any such an identification $M_1\ovt M_2 = M_1^t\ovt  M_2^{1/t}$. 
So we do not need to be careful to identify $M_1\ovt M_2$ with $M_1^t\ovt  M_2^{1/t}$ in the study of Popa's conjugacy.
\end{Rem}

Here we record a simple but very useful lemma on tensor product factors.

\begin{Lem}\label{commutant form}
	Let $M_1\ovt  M_2= N_1\ovt  N_2$ $(=:M)$ be a tensor decomposition as $\rm II_1$ factors and assume that $M_1\subset N_1$. Then $M_1'\cap N_1$ is a factor and satisfies  
$$M_2=(M_1'\cap N_1)\ovt  N_2 \quad \text{and} \quad N_1=M_1\ovt  (M_1'\cap N_1).$$
\end{Lem}
\begin{proof}
	Since $M_1\subset N_1$, we have 
$$M_2=M_1'\cap M=M_1'\cap (N_1\ovt N_2)=(M_1'\cap N_1)\ovt  N_2.$$ 
So $M_1'\cap N_1$ is a factor. We have $M=M_1\ovt  (M_1'\cap N_1)\ovt  N_2$ and hence $N_1=N_2'\cap M=M_1\ovt  (M_1'\cap N_1)$.
\end{proof}

	The following lemma is a key observation in this paper, which states a sufficient condition to the property (TFF) in terms of Popa's conjugacy. Although its proof is easy, this lemma plays significant roles in our study.

\begin{Lem}\label{tensor formula}
	Let $M$ be a prime $\rm II_1$ factor satisfying the following condition.
\begin{itemize}
	\item For any $\rm II_1$ factor $B$ and any $t>0$ such that $M \ovt B \simeq M \ovt B^t \ (=:K\ovt L)$, under the identification $M \ovt B=K\ovt L$, we have either 
	$$K\preceq_{M \ovt B} B, \quad L\preceq_{M \ovt B} B, \quad M \preceq_{M \ovt B}L,\quad \text{or}\quad B\preceq_{M \ovt B}L.$$
\end{itemize}
Then $M$ has the property (TFF).
\end{Lem}
\begin{proof}
	Fix a $\rm II_1$ factor $B$ and take $t\in \mathcal{F}(M \ovt B)$. We will show $t \in \mathcal{F}(M)\mathcal{F}(B)$. Fix an isomorphism $M \ovt  B \simeq (M \ovt  B)^t\simeq M \ovt B^t$ $(=:K\ovt  L)$. By assumption, regarding $N:=M \ovt B^t=K\ovt L$, we have either 
	$$K\preceq_N B, \quad L\preceq_N B, \quad M \preceq_NL,\quad \text{or}\quad B\preceq_NL.$$

	If $L\preceq_NB$, then by Lemma \ref{intertwining tensor} there exists $s>0$ and $u\in\mathcal{U}(N)$ such that $uLu^*\subset B^s$ under the isomorphism $M\ovt B = M^{1/s} \ovt B^s$. For simplicity we assume $u=1$. Then by Lemma \ref{commutant form}, 
 putting $P :=  L' \cap B^s$, it holds that 
	$$B^s = L \ovt P \quad \text{and} \quad K =  P \ovt M^{1/s}.$$
Since $K(=M)$ is prime, $P$ is finite dimensional. Write $P=\M_n$ for some $n\in\N$ and we obtain
	$$B^s = L^n  \quad \text{and} \quad K = M^{n/s}.$$
Since $L^n = B^{tn}$ and $K=M$, this implies that $s/tn \in\mathcal{F}(B)$ and $n/s \in\mathcal{F}(M)$, and hence 
	$$\frac{1}{t} = \frac{s}{tn} \cdot \frac{n}{s} \in  \mathcal{F}(B)\, \mathcal{F}(M).$$
Thus $t \in \mathcal{F}(B)\, \mathcal{F}(M)$.

	Next assume $K\preceq_NB$. Then by the same reasoning as above, there exists $s>0$ and $u\in \mathcal{U}(N)$ such that $uKu^* \subset B^s$. We assume $u=1$. Putting $Q :=  K' \cap B^s$ it holds that 
	$$B^s = K \ovt Q \quad \text{and} \quad L =  Q \ovt M^{1/s}.$$
Since $K=M$ and $L=B^t$, these equations imply
	$$ B^s = M \ovt Q \quad \text{and} \quad B^t = Q \ovt M^{1/s},$$
and hence
	$$B^s = M \ovt Q \simeq Q\ovt M = B^{ts}.$$
This implies $t \in \mathcal{F}(B)$ and we obtain the conclusion.

	Finally assume that $M \preceq_NL$ or $B \preceq_NL$. Then since $K=M$ and $L=B^t$, if we put $\widetilde{B}:=B^t$, $\widetilde{t}:=1/t$, $\widetilde{K}:=M$, $\widetilde{L}:=\widetilde{B}^{\widetilde{t}}$, and $\widetilde{K}\ovt\widetilde{L}=M \ovt \widetilde{B}$, we can apply exactly the same argument as in the previous two cases, and obtain $\widetilde{t} = 1/t \in \mathcal{F}(M) \mathcal{F}(\widetilde{B})$. Since $\mathcal{F}(\widetilde{B})=\mathcal{F}(B^t)=\mathcal{F}(B)$, we obtain the conclusion.
\end{proof}

\subsection*{\bf Strong primeness}

	We study fundamental properties on strong primeness. We first give a reformulation of strong primeness in terms of Popa's conjugacy.

\begin{Lem}\label{lemma}
	A $\rm II_1$ factor $M$ is strongly prime if and only if for any tensor decomposition $M\ovt B=K\ovt  L$ as $\rm II_1$ factors, we have either $K\preceq B$ or $L\preceq B$.
\end{Lem}
\begin{proof}
	Use Lemma \ref{intertwining tensor} and \cite[\textrm{Lemma 3.5}]{Va08}.
\end{proof}

	We deduce primeness from strong primeness. This is not entirely trivial since, in the definition of strong primeness, we mention only a decomposition as $\rm II_1$ factors. 

\begin{Pro}\label{s-prime to TF}
	Strong primeness implies primeness. In particular any strongly prime $\rm II_1$ factor satisfies the property (TFF).
\end{Pro}
\begin{proof}
	Let $M$ be a non-prime $\rm II_1$ factor with a decomposition $M=M_1\ovt  M_2$ as $\rm II_1$ factors. Fix any $\rm II_1$ factor $B$ and put $K:=M_1$, $L:=M_2\ovt  B$, and $N:=M\ovt  B=K\ovt  L$. 
Then if $M$ is strongly prime, we have either $M\preceq_NK$ or $M\preceq_NL$. 
By Lemma \ref{tensor case}, the first one is equivalent to $M_2\preceq_{M_2}\mathbb{C}$ and the second one is to $M_1\preceq_{M_1}\mathbb{C}$. Thus in each case, we get a contradiction. 
Use Lemmas \ref{tensor formula} and \ref{lemma} for the second assertion. 
\end{proof}

	We observe the difference between the two notions of primeness discussed above. This follows from \cite[Theorem B]{Ho15}.

\begin{Pro}\label{prime not to s-prime}
	Any strongly prime $\rm II_1$ factor is full. In particular there is a prime $\rm II_1$ factor, which is not strongly prime.
\end{Pro}
\begin{proof}
	Let $M$ and $B$ be non-full $\rm II_1$ factors. Then by \cite[Theorem B]{Ho15}, there is an automorphism $\phi$ on $M\ovt B$ such that $\phi(M)\not\preceq_{M\ovt B} B$ and $\phi(B)\not\preceq_{M\ovt B} B$. Thus the decomposition $M\ovt B = \phi(M)\ovt \phi(B)$ shows that $M$ is not strongly prime.

	 Let $\F_2 \curvearrowright X$ be a free, ergodic, and measure preserving action of the free group on a standard probability space. Assume that it is not strongly ergodic. Then the crossed product $M:=L^\infty(X)\rtimes \F_2$ is a prime $\rm II_1$ factor by \cite[Theorem 4.6]{Oz04}, and is not strongly prime since it is not full. 
\end{proof}

\section{\bf Proof of Proposition \ref{prime factorization}}

	We study a \textit{unique prime factorization} phenomena, by using our strong primeness. This was already mentioned in our previous paper \cite[Corollary 5.1.3]{Is14}, that shows strongly prime factors behave like prime numbers with respect to von Neumann algebra tensor products. 
We only discussed the case of tensor products with finitely many strongly prime factors. So in this paper, we study the case of infinite tensor products. 

	 We start with several lemmas.

\begin{Lem}\label{lemma1}
	Let $M$ be a strongly prime $\rm II_1$ factor and let $M\ovt  B=N_1\ovt \cdots\ovt  N_n$ $(=:N)$ be a tensor decomposition as $\rm II_1$ factors with $n\geq 2$. 
Then there is $i$ such that $M\preceq_NN_i$.
\end{Lem}
\begin{proof}
	We prove it by induction on $n$. The case $n=2$ is obvious by the definition of strong primeness. So assuming $n-1 \geq 2$ is proven, we show the case $n$ holds. 

	Put $N_1':=N_1'\cap N$. Then since $M\ovt  B=N_1\ovt  N_1'$, we have either $M\preceq_NN_1$ or $M\preceq_N N_1'$. Since $M\preceq_N N_1$ implies the conclusion, we may assume $M\preceq_NN_1'$. By Lemma \ref{intertwining tensor} we find $u\in\mathcal{U}(N)$ and $t>0$ such that $uMu^*\subset (N_1')^t$. Then by Lemma \ref{commutant form} we have $uMu^*\ovt  P= (N_1')^t$, where $P=(uMu^*)'\cap (N_1')^t$. 
Observe that $P$ is a $\rm II_1$ factor. In fact, if $P$ is finite dimensional, then  because $M$ is prime, $n$ must be 2 which contradicts our assumption. 

	Now we can apply strong primeness of $M$ and the assumption on the induction to the decomposition $uMu^* \ovt P = N_2 \ovt \cdots \ovt N_n$ and get that $u M u^* \preceq_{uMu^* \ovt P} N_i$ for some $i \geq 2$. Then take $\theta, p,q,v$ as in Theorem \ref{Popa embed}(i), and observe that $\theta \circ \Ad u, u^*pu ,q, u^*v$ gives the condition $M \preceq_N N_i$. Thus we get the conclusion.
\end{proof}

\begin{Lem}\label{lemma2}
	Let $M\ovt  B=N_1\ovt \cdots\ovt  N_n$ $(=:N)$ be a tensor decomposition as $\rm II_1$ factors with $n\geq 2$. If $M\preceq_NN_i$ and $M\preceq_NN_j$, then $i=j$.
\end{Lem}
\begin{proof}
	Suppose by contradiction that $i\neq j$, and put $i=1$ and $j=2$ for simplicity. Then by Lemmas \ref{intertwining tensor} and \ref{commutant form}, one has $uMu^*\ovt  P=N_1^t$, where $u\in\mathcal{U}(N)$, $t>0$, and $P:=(uMu^*)'\cap N_1^t$, that gives a decomposition 
	$$N=uMu^*\ovt  P\ovt  N_2^{1/t}\ovt  N_3\ovt  \cdots\ovt  N_n .$$
Observe by Lemma \ref{intertwining tensor} that the given condition $M\preceq_N N_2$ is equivalent to  $uMu^*\preceq_N N_2^{1/t}$. By Lemma \ref{tensor case}, we get $uMu^*\preceq_{N\cap (N_2^{1/t})'}\mathbb{C}$, which contradicts the diffuseness of $M$.
\end{proof}

\begin{Lem}\label{lemma3}
	Let $M_1\ovt  M_2\ovt  B=N_1\ovt \cdots\ovt  N_n$ $(=:N)$ be a tensor decomposition as $\rm II_1$ factors with $n\geq 2$. If $M_1\preceq_NN_1$ and $M_2\preceq_NN_1$, then $M_1\ovt  M_2\preceq_NN_1$. In this case, $N_1$ is not prime.
\end{Lem}
\begin{proof}
	The first assertion is immediate by Lemma \ref{Sa09}. 
For the second one, by Lemmas \ref{intertwining tensor} and \ref{commutant form}, take $u\in\mathcal{U}(N)$ and $t>0$ such that $u(M_1\ovt  M_2)u^*\subset N_1^t$ and $N_1^t=u(M_1\ovt  M_2)u^*\ovt  P$, where $P:=u(M_1\ovt  M_2)'u^*\cap N_1^t$. Since $N_1\simeq M_1^{1/t} \ovt  M_2\ovt  P$, $N_1$ is not prime.
\end{proof}

\begin{Lem}\label{lemma5}
	Let $\ovt_{i=1}^m M_i= N\ovt B$ $(=:M)$ be a tensor decomposition as $\rm II_1$ factors with $m=\infty$. If $B \preceq_M \ovt_{i=k}^m M_i$ for all $k \in \N$, then $B$ is amenable.
\end{Lem}
\begin{proof}
	We follow the idea in \cite[Proposition 4.2]{HU15} due to Ioana. In the proof, for any subset  $\mathcal{F}\subset \N$ we put $M_\mathcal{F}:=\ovt_{i\in \mathcal{F}} M_i \subset M$.

Put $\mathcal{M}:= M \ovt M$ and we regard the left $M$ as the original one. Let $\Sigma$ be the flip map on $\mathcal{M}$ given by $\Sigma(a\otimes b) = b\otimes a$. For any $\mathcal{F} \subset \N$, put $\mathcal{M}_{\mathcal{F}}:=M_{\mathcal{F}} \ovt M_{\mathcal{F}}$ with the flip $\Sigma_{\mathcal{F}}$. We regard $\Sigma_{\mathcal{F}}\in \mathrm{Aut}(\mathcal{M})$ by putting $\Sigma_{\mathcal{F}}|_{\mathcal{M}_{\mathcal{F}^c}} := \id$.  
Observe that $\text{weak-}\lim_{\mathcal{F}}\Sigma_{\mathcal{F}}(x) = \Sigma(x)$ for all $x \in \mathcal{M}$, where the limit is taken over all \textit{finite} subsets $\mathcal{F}\subset \N$. 

	Observe next $B \preceq_M M_{\mathcal{F}^c}$ for any finite $\mathcal{F}\subset \N$ by assumption, so there is a unitary $v_{\mathcal{F}} \in M$ and $t_{\mathcal{F}}>0$ such that $v_{\mathcal{F}} B v_{\mathcal{F}}^* \subset M_{\mathcal{F}^c}^{t_{\mathcal{F}}}$ by Lemma \ref{intertwining tensor}. In this case, we may assume that $M_{\mathcal{F}^c}^{t_{\mathcal{F}}} \subset  M_{\max{\mathcal{F}}}\ovt M_{\mathcal{F}^c} $ where $\max{\mathcal{F}}:=\max\{i\mid i\in \mathcal{F}\}$ (recall that we are fixing $M= M_{\mathcal{F}}^{1/t_{\mathcal{F}}} \ovt M_{\mathcal{F}^c}^{t_{\mathcal{F}}}$, so applying again a partial isometry conjugacy at the level of $M\ovt \M_n$ for some $n\in\N$ we may assume this condition). 
In particular  we have $\Sigma_{\mathcal{F}^\circ} (v_{\mathcal{F}} b v_{\mathcal{F}}^* \otimes 1 ) = v_{\mathcal{F}} b v_{\mathcal{F}}^* \otimes 1$ for all $b\in B$, where $\mathcal{F}^\circ := \mathcal{F}\setminus \max\mathcal{F}$. 
We put $u_{\mathcal{F}}:=(v_{\mathcal{F}}^*\otimes 1) \Sigma_{\mathcal{F}^\circ}(v_{\mathcal{F}} \otimes 1) \in \mathcal{U}(\mathcal{M})$ and calculate that for all $b\in B$,
\begin{eqnarray*}
	u_{\mathcal{F}} \Sigma_{\mathcal{F}^\circ}(b\otimes 1) 
	&=& (v_{\mathcal{F}}^*\otimes 1) \Sigma_{\mathcal{F}^\circ}(v_{\mathcal{F}}\otimes 1) \Sigma_{\mathcal{F}^\circ}(b \otimes 1) \\
	&=& (v_{\mathcal{F}}^*\otimes 1) \Sigma_{\mathcal{F}^\circ}(v_{\mathcal{F}}b \otimes 1) \\
	&=& (v_{\mathcal{F}}^*\otimes 1) \Sigma_{\mathcal{F}^\circ}(v_{\mathcal{F}}bv_{\mathcal{F}}^* \otimes 1) \Sigma_{\mathcal{F}^\circ}(v_{\mathcal{F}} \otimes 1)\\
	&=& (v_{\mathcal{F}}^*\otimes 1) (v_{\mathcal{F}}bv_{\mathcal{F}}^* \otimes 1) \Sigma_{\mathcal{F}^\circ}(v_{\mathcal{F}} \otimes 1)\\
	&=& (b \otimes 1) u_{\mathcal{F}}.
\end{eqnarray*}
Define a state $\Omega$ on $\B(L^2(\mathcal{M}))$ by $\Omega(X):= \lim_{\mathcal{F}}\langle X u_{\mathcal{F}}, u_{\mathcal{F}}\rangle_{L^2(\mathcal{M})}$, where the limit is taken over all finite $\mathcal{F}$. It satisfies for $x\in \mathcal{M}$ 
	$$\Omega(a) = \lim_{\mathcal{F}}\langle a u_{\mathcal{F}}, u_{\mathcal{F}}\rangle_{L^2(\mathcal{M})} = \lim_{\mathcal{F}}\tau_{\mathcal{M}}( u_{\mathcal{F}}^* a u_{\mathcal{F}}) = \tau_{\mathcal{M}}(a).$$
For all $b\in \mathcal{U}(B)$, regarding $L^2(\mathcal{M}) = L^2(M)\otimes L^2(M)$ with the right $M$-action given by $M\ni x \mapsto 1\otimes J_M x^* J_M$ where $J_M$ is the anti-unitary map $J_M(y) = y^*$ for $y\in M \subset L^2(M)$, since $u_{\mathcal{F}} \Sigma_{\mathcal{F}^\circ}(b\otimes 1) = (b\otimes 1)u_{\mathcal{F}}$ and $\Sigma_{\mathcal{F}^\circ}(b\otimes 1) \to \Sigma(b\otimes 1)$ weakly for all $b\in B$, we have
\begin{eqnarray*}
	\Omega(b\otimes J_MbJ_M) 
	&=& \lim_{\mathcal{F}}\langle (b\otimes J_MbJ_M)u_{\mathcal{F}}, u_{\mathcal{F}}\rangle_{L^2(\mathcal{M})} \\
	&=& \lim_{\mathcal{F}}\langle (b\otimes 1)u_{\mathcal{F}} (1\otimes b^*), u_{\mathcal{F}}\rangle_{L^2(\mathcal{M})} \\
	&=& \lim_{\mathcal{F}}\langle u_{\mathcal{F}} \Sigma_{\mathcal{F}^\circ}(b\otimes 1) (1\otimes b^*), u_{\mathcal{F}}\rangle_{L^2(\mathcal{M})} \\
	&=& \lim_{\mathcal{F}} \tau_{\mathcal{M}} (\Sigma_{\mathcal{F}^\circ}(b\otimes 1) (1\otimes b^*)) \\
	&=& \tau_{\mathcal{M}} (\Sigma(b\otimes 1) (1\otimes b^*)) =1 .
\end{eqnarray*}
So the state $\Omega$ satisfies $\Omega((b\otimes J_MbJ_M)(X\otimes 1)) = \Omega(X\otimes 1)$ and hence $\Omega(bXb^*\otimes 1) = \Omega((b\otimes J_MbJ_M)(X\otimes 1)(b\otimes J_MbJ_M)^* ) = \Omega (X)$ for all $X\in \B(L^2(M))$ and $b\in \mathcal{U}(B)$. Thus the restriction of $\Omega$ on $\B(L^2(M))\otimes \C1_{L^2(M)}$ is a $B$-central state which is the trace on $B$. This means $B$ is amenable.
\end{proof}

\begin{proof}[Proof of Proposition \ref{prime factorization}]
	We fix $i\in \N$ with $1\leq i \leq m$. Then since $M_i$ is non-amenable, strong primeness and Lemma \ref{lemma5} imply that there is $k \in \N$ with $1\leq k \leq n$ such that $M_i \preceq_M \ovt_{j=1}^k N_j$ (this is obvious if $n \neq \infty$). 
By Lemmas \ref{intertwining tensor} and \ref{commutant form}, one has $uM_iu^*\ovt  P = N_1^t \ovt N_2 \ovt \cdots \ovt N_k$ for a factor $P$, $u\in\mathcal{U}(M)$, and $t>0$. 
Then if $P$ is of type I, then $k=1$ by the primeness of $M_i$ and hence $M_i\preceq_M N_1$. If $P$ is a $\rm II_1$ factor, then by Lemma \ref{lemma1} there is some $j\in \N$ with $1\leq j \leq k$ such that $M_i \preceq_M N_j$. Thus in any case there is $j$ such that $M_i \preceq_M N_j$. We put $\sigma(i):=j$, and $\sigma$ is uniquely determined by Lemma \ref{lemma2}.

\bigskip
\noindent
{\bf Surjectivity of $\sigma$.}

	Assume that $\sigma$ is surjective. Then since the condition $M_i \preceq_M N_{\sigma(i)}$ implies non-amenability of $N_{\sigma(i)}$, we have that all $N_j$ are non-amenable. 

To see the converse direction, we show the following claim.
\begin{claim1}
	Assume that there is $j_0 \in \N$ with $1\leq j_0 \leq n$ such that $M_i \not\preceq_M N_{j_0}$ for all $i\in \N$ with $1\leq i\leq n$. Then we have $\rm (i)$ a contradiction if $m\neq\infty$, and $\rm (ii)$ $N_{j_0}$ is amenable if $m=\infty$. 
\end{claim1}
\begin{proof}
	We fix $k \in \N$ with $1\leq k \leq m$. Observe that $M_i \preceq_M \ovt_{j=1}^k N_{\sigma(j)} $ for all $1\leq i\leq k$, and then Lemma \ref{lemma3} implies $\ovt_{i=1}^k M_i \preceq_M \ovt_{j=1}^k N_{\sigma(j)}$. 
By taking relative commutants, we have 
	$$N_{j_0} \subset (\ovt_{j=1}^k N_{\sigma(j)})' \cap M \preceq_M (\ovt_{i=1}^k M_i)'\cap M = \ovt_{i=k+1}^m M_i.$$ 
If $m\neq \infty$, one can put $k=m$ and obtain $N_{j_0} \preceq_M \C$, a contradiction. If $m=\infty$, then we have $N_{j_0} \preceq_M \ovt_{i=k+1}^m M_i$ for all $k \in \N$ that implies amenability of $N_{j_0}$ by Lemma \ref{lemma5}. 
\end{proof}

	Observe now that $j_0 \not\in\mathrm{Im}\sigma$ if and only if $M_i \not\preceq_M N_{j_0}$ for all $i\in \N$ with $1\leq i\leq n$ (since $M_i \preceq_M N_{j_0}$ exactly means $\sigma(i)=j_0$ by the uniqueness of $\sigma$). 
So this claim shows that $\rm (i)$ $\sigma$ is always surjective if $m\neq \infty$, and $\rm (ii)$ if $m=\infty$ non-surjectivity of $\sigma$ implies amenability of $N_{j_0}$ for some $j_0$. 
This completes the statement for surjectivity (note that $N_j$ can not be amenable if $m \neq \infty$ as we mentioned in Introduction). 

\bigskip
\noindent
{\bf Injectivity of $\sigma$.}

	Assume next that $\sigma$ is not injective. Then there are $i\neq i'$ such that $\sigma(i)= \sigma(i')=:j$, that means $M_i \preceq_M N_{j}$ and $M_{i'}\preceq_MN_j$. By (the proof of) Lemma \ref{lemma3}, $N_j$ is isomorphic to $M_i^t \ovt M_{i'}\ovt P$ for some $t>0$ and a factor $P$. Since $M_i$ and $M_{i'}$ are non-amenable, $N_j$ is not semiprime. 

	Conversely assume $N_{\sigma(i)}$ is not semiprime for some $i$, so there is a tensor decomposition $N_{j_0}=N_{j_0}^1 \ovt N_{j_0}^2$ with non-amenable II$_1$ factors $N_{j_0}^1$ and $N_{j_0}^2$. 
If $M_i \not\preceq_M N_{j_0}^1$ for all $i$, then the claim above shows that we have (i) a contradiction if $m\neq \infty$, and (ii) $N_{j_0}^1$ is amenable if $m=\infty$. So non-amenability of $N_{j_0}^1$ and $N_{j_0}^2$ implies there is $k,l\in \N$ such that $M_k \preceq_MN_{j_0}^1$ and $M_l \preceq_MN_{j_0}^2$. We know $k\neq l$ by Lemma \ref{lemma2}. Finally since $N_{j_0}^1,N_{j_0}^2 \subset N_{j_0}$, we have $M_k \preceq_MN_{j_0}$ and $M_l \preceq_MN_{j_0}$ that means $\sigma(k)=\sigma(l)$. So $\sigma$ is not injective.

\bigskip

	Finally we assume that each $N_j$ is non-amenable and semiprime. Then  $\sigma$ is bijective by previous arguments. By Lemmas \ref{intertwining tensor} and \ref{commutant form}, $M_i \preceq_M N_{\sigma(i)}$ implies $N_{\sigma(i)}\simeq M_i^{t_i} \ovt P_i$ for some $t_i>0$ and a factor $P_i$. Since $N_{\sigma(i)}$ is semiprime, $P_i$ must be amenable. 
\end{proof}

\section{\bf Proofs of main theorems}

	In the proofs of main theorems, we will make use of the following three structural theorems. Note that all of them are formulated with relative amenability, and this relativity is crucial to our proofs.

\begin{Thm}[{\cite[\textrm{Theorem 1.4}]{PV12}}]\label{PV12}
	Let $B$ be any finite von Neumann algebra and $\Gamma$ be weakly amenable and bi-exact group acting on $B$. 
Put $M:=B\rtimes\Gamma$. Then for any von Neumann subalgebra $A\subset M$ which is amenable relative to $B$ in $M$, we have either $\rm (i)$ $A\preceq_MB$ or $\rm (ii)$ $\mathcal{N}_M(A)''$ is amenable relative to $B$ in $M$.
\end{Thm}

\begin{Thm}[{\cite[Theorem 1.6]{Io12}\cite[\textrm{Theorem A}]{Va13}}]\label{Va13}
	Let $M = M_1*_BM_2$ be an amalgamated free product of tracial von Neumann algebras $(M_i, \tau)$ with common von Neumann subalgebra $B \subset M_i$ w.r.t. the unique trace preserving conditional expectations. Let $A \subset M$ be a von Neumann subalgebra that is amenable relative to $B$ inside $M$ and satisfies $A \not\preceq_M B$. 
Then we have either $\rm (i)$ $\mathcal{N}_M(A)'' \preceq_M M_i$ for some $i$ or $\rm (ii)$ $\mathcal{N}_M(A)''$ is amenable relative to $B$ inside $M$.
\end{Thm}

\begin{Thm}[{\cite[\textrm{Theorem 2.2}]{SW11}}]\label{SW11}
	Let $\Gamma$ be a wreath product group of a non-trivial amenable group by a non-amenable group and let $B$ be a finite von Neumann algebra. Put $M:=B\ovt  L\Gamma$. Let $Q\subset M$ be a von Neumann subalgebra which is not amenable relative to $B$. If $Q'\cap M$ is a regular subfactor in $M$, then we have $Q'\cap M\preceq_MB$.
\end{Thm}

\begin{proof}[Proof of Theorem \ref{thmC}]
	The first case was already proved in \cite[Theorem 5.1.1]{Is14}.

	Suppose by contradiction that $M$ is not strongly prime. Then by Lemma \ref{lemma}, there are $\rm II_1$ factors $B,K$, and $L$ such that $B\ovt  M=K\ovt  L$ $(=:N)$ with $K\not\preceq_N B$ and $L\not\preceq_N B$.

\begin{case1}
	{\bf $\bf M$ is a free product $\bf M_1*M_2$.} 
\end{case1}

	By \cite[Corollary F.14]{BO08}, there is a diffuse abelian subalgebra $A\subset K$ such that $A\not\preceq_NB$. Then regarding $N=(B\ovt  M_1)*_B(B\ovt  M_2)$, we apply Theorem \ref{Va13} to $A\subset N$ and get either (i) $\mathcal{N}_N(A)'' \preceq_N (B\ovt  M_i)$ for some $i$ or (ii) $\mathcal{N}_N(A)''$ is amenable relative to $B$ in $N$. 

	Assume first that (ii) happens. Since $L\subset\mathcal{N}_N(A)''$, $L$ is amenable relative to $B$ in $N$. By Theorem \ref{Va13}, we get either $\rm (i)'$ $N \preceq_N (B\ovt  M_i)$ for some $i$ or $\rm (ii)'$ $N$ is amenable relative to $B$ inside $N$. 
If $\rm (i)'$, by Lemma \ref{tensor case}, one has $M\preceq_MM_i$ which contradicts diffuseness of $M_j$ (where $i\neq j$) by Lemma \ref{free product case}. If $\rm (ii)'$, then we get that $M$ is amenable by Lemma \ref{relative amenable and amenable}, which is a contradiction. Thus the condition (ii) does not happen.

	Assume next condition (i). We have two conditions $L \preceq_N (B\ovt  M_i)$ and $L\not\preceq_NB$, and it is known that they imply  $N=K\ovt L\preceq_N(B\ovt  M_i)$ \cite{IPP05}. Here we give a sketch of this argument in the paragraphs below for reader's convenience. Once we obtain it, then by Lemma \ref{tensor case}, this means $M\preceq_MM_i$ which contradicts diffuseness of $M_j$ (where $i\neq j$) by Lemma \ref{free product case}, and hence we can end the proof.

	Suppose now that $L \preceq_N (B\ovt  M_1)$. Then there is a $*$-homomorphism $\theta\colon pLp\rightarrow q(B\ovt  M_1)q$ for some projections $p\in L$, $q\in B\ovt  M_1$, and a partial isometry $v\in N$ such that $v\theta(x)=xv$ for $x\in pLp$. We may replace $q$ with the support projection of $E_{B\ovt  M_1}(v^*v)$. Put $D:=\theta(pLp)$. If $D\preceq_{B\ovt  M_1} B$, then by the choice of $q$, we can deduce $L \preceq_{B\ovt M} B$ (e.g.\ \cite[Remark 3.8]{Va08}) and hence a contradiction. So we have $D\not\preceq_{B\ovt  M_1} B$.

	By \cite[Theorem 1.1]{IPP05}, any quasi-normalizer of $D$ in $q(B\ovt  M)q$ is contained in $B\ovt  M_1$. In particular we have $v^*v\in B\ovt  M_1$. We put $\tilde{q}:=v^*v$, $\tilde{\theta}:=\theta(\cdot)\tilde{q}$, and $\widetilde{D}:=D\tilde{q}$, and observe that $\widetilde{D}\not\preceq_{B\ovt  M_1} B$. Write $vv^*=pp'$ for some $p'\in L'\cap N=K$. Then we get a $\ast$-homomorphism $\mathrm{Ad}v^*\colon pLpp'\rightarrow \tilde{q}(B\ovt  M_1)\tilde{q}$. Since $v^*pLpp'v=\widetilde{D}\not\preceq_{B\ovt  M_1} B$, again by \cite[Theorem 1.1]{IPP05}, any quasi-normalizer of $v^*pLpp'v$ is contained in $B\ovt M_1$. Hence we have $v^*pp'Kpp'v \subset \tilde{q}(B\ovt  M_1)\tilde{q}$. Thus we obtain $v^*pp'(K\ovt L)pp'v \subset \tilde{q}(B\ovt  M_1)\tilde{q}$ and $K\ovt L\preceq_N B\ovt M_1$. This is the desired condition.

\begin{case1}
	{\bf $\bf M$ is a wreath product group factor $\bf L(\Delta\wr\Lambda)$.} 
\end{case1}

	Write $\Lambda=\Lambda_1\times \Lambda_2$ as in the statement. For simplicity, we also write as $\Gamma:=\Delta\wr\Lambda=\Delta_\Lambda\rtimes\Lambda$ and $\Gamma_i:=\Delta_\Lambda \rtimes {\Lambda}_i$ for all $i$.

	Since $K$ and $L$ are regular subfactors, and $K=L'\cap N$ and $L=K'\cap N$, by Theorem \ref{SW11}, it holds that $K$ and $L$ are amenable relative to $B$ in $M$. In particular, $K$ and $L$ are amenable relative to $B\ovt  L{\Gamma}_1$ (since $B \subset B\ovt  L{\Gamma}_1$). Regarding $B \ovt L\Gamma$ as a crossed product of $B\ovt  L{\Gamma}_1$ by $\Lambda_2$, by Theorem \ref{PV12}, we get $K\preceq_N B\ovt  L{\Gamma}_1$ and $L\preceq_N B\ovt  L{\Gamma}_1$. 
We can then apply Lemma \ref{Sa09} and obtain that $N = K \vee L \preceq_N B\ovt \Gamma_1$. However by Lemma \ref{crossed product case}, this contradicts the fact that $\Lambda_2$ is an infinite group. 
\end{proof}

\begin{proof}[Proof of Theorem \ref{thmA}]
	We consider only the case that $M$ is the wreath product group factor, and other cases are proved by Theorem \ref{thmC} and Proposition \ref{s-prime to TF}.

	Put $\Gamma:=\Delta\wr \Lambda$ and $M:=L\Gamma$. We will verify the sufficient condition in Lemma \ref{tensor formula}. Let $B$ be a $\rm II_1$ factor and $t>0$ such that $M\ovt B\simeq M \ovt B^t$ $(=:K\ovt L)$. Regarding $M\ovt B=K\ovt L$, we will show that either $K\preceq_{M\ovt B}B$, $L\preceq_{M\ovt B}B$, $M\preceq_{M\ovt B}L$, or $B\preceq_{M\ovt B}L$. So suppose by contradiction that any of them does not hold and we will deduce amenability of $M$, which is a contradiction.

	We apply Theorem \ref{SW11} to $K$ and get either (i) $K$ is amenable relative to $B$ in $M\ovt B$ or (ii) $K'\cap (M\ovt B)=L\preceq_{M\ovt B} B$. So by assumption, we have that $K$ is amenable relative to $B$ in $M\ovt B$. By the same reason, $L$ is also amenable relative to $B$ in $M\ovt B$. Exchanging the roles of $M\ovt B$ and $M\ovt B^t$, it further holds that $M$ and $B$ are amenable relative to $L$ in $M\ovt B$. Hence using $M \lessdot_{M\ovt B}L$ and $L\lessdot_{M\ovt B}B$ together with Proposition \ref{relative amenable transitivity}, we obtain that $M$ is amenable relative to $B$ in $M\ovt B$. This means that $M$ is amenable by Lemma \ref{relative amenable and amenable} and thus a contradiction.
\end{proof}

\section{\bf Proof of Corollary \ref{corB}}

	Let $G_n \in \mathcal{S}_{\rm factor}$ for $n\in \N$ (possibly $G_n=G_m$ for different $n,m$), and take II$_1$ factors $B_n$ with separable predual such that $\mathcal{F}(B_n)=G_n$. We may assume $B_n=B_m$ whenever $G_n=G_m$. 
Let $N$ be a free product $\rm II_1$ factor given by $N:=L\F_2 * L(\Z^2\rtimes \mathrm{SL}(2,\Z))$. Observe that $\mathcal{F}(N)=\{1\}$ by \cite[Corollary 6.4]{IPP05} and hence $\mathcal{F}(N\ovt B_n) = \mathcal{F}(B_n)$ by Theorem \ref{thmA}. Define an infinite free product $\rm II_1$ factor $M:= *_{n=1}^\infty M_n$, where $M_n:=N\ovt B_n$ for all $n\in \N$. We first show that it satisfies $\mathcal{F}(M)=\bigcap_{n\in\N}\mathcal{F}(B_n)=\bigcap_{n\in \N}G_n$. 

	Recall first from \cite[Theorem 1.5]{DR99} that for any $0<t \leq 1$, one has 
	$$M^t = *_{n=1}^\infty M_n^t$$ 
and this implies $\bigcap_{n\in\N} \mathcal{F}(M_n) \subset  \mathcal{F}(M)$ \cite[Corollary 1.6]{DR99}. Since $\mathcal{F}(B_n) = \mathcal{F}(M_n)$ for all $n\in \N$, we get an inclusion $\bigcap_{n\in\N}\mathcal{F}(B_n) \subset \mathcal{F}(M) $. 

	We next see the reverse inclusion. Fix $t \in \mathcal{F}(M)$. Up to replacing with $1/t$ if necessary, we may assume $0<t \leq 1$ and so we have an isomorphism 
	$$*_{n=1}^\infty M_n = M \simeq M^t =*_{n=1}^\infty M_n^t.$$
Since each $M_n$ is a tensor product of non-amenable $\rm II_1$ factors, we can apply \cite[Main Theorem]{HU15} (see also \cite{Oz04,IPP05,Po06b} for the case of  finitely many free components). So there is a bijection $\alpha$ on $\N$ such that $M_n$ and $M_{\alpha(n)}^t$ are isomorphic for all $n\in \N$. 
Indeed, \cite[Main Theorem]{HU15} actually shows $M_n\preceq_M M_{\alpha(n)}^t$ and $M_{\alpha(n)}^t \preceq_M M_n$. Once we get this condition, then by the proof of unique factorization of free products $\rm II_1$ factors (e.g.\ \cite[Theorem 3.3]{Oz04}) one can show that $M_n$ and $M_{\alpha(n)}^t$ are unitary conjugate in $M$, namely, there is $u\in \mathcal{U}(M)$ such that $uM_n u^* = M_{\alpha(n)}^t$ (under the given isomorphism). This particularly implies 
	$$G_n=\mathcal{F}(B_n)=\mathcal{F}(M_n)=\mathcal{F}(M_{\alpha(n)}^t)=\mathcal{F}(M_{\alpha(n)})=\mathcal{F}(B_{\alpha(n)})=G_{\alpha(n)}$$
and hence $B_n=B_{\alpha(n)}$ by our choice of $\{B_k\}_{k\in\N}$. Thus the above isomorphism $M_n \simeq M_{\alpha(n)}^t$ means $t\in \mathcal{F}(M_n)=\mathcal{F}(B_n)$ for each $n\in \N$, and so $t\in \bigcap_{n\in\N}\mathcal{F}(B_n)$. We conclude $\mathcal{F}(M) \subset \bigcap_{n\in\N}\mathcal{F}(B_n)$.

	Now we start the proof of Corollary \ref{corB}. The stability for intersection was already proved above. 
Let $G\in\mathcal{S}_{\rm factor}$ and take a II$_1$ factor $B$ with separable predual such that $\mathcal{F}(B)=G$. Then by putting $B_n:=B$ for all $n\in \N$, the above argument shows that $\mathcal{F}(B) = \mathcal{F}(M)$ for $M:=\ast_{n\in\N} (B\ovt N)$, which is exactly the formula we mentioned in Introduction.
(Note that even in the case $B_n=B$ for all $n\in \N$, one needs infinitely many free product components, since \cite[Theorem 1.5]{DR99} holds only for infinite free products.)
Since $M$ is a free product, it satisfies the property (TFF), so the first assertion of Corollary \ref{corB} holds. 
The stability for multiplication is then an immediate consequence of the first assertion and the definition of the property (TFF).

\section{\bf Some partial results}

	It would be interesting to know whether $L(\Z^2\rtimes \mathrm{SL}(2,\Z))$ satisfies the property (TFF) or not. However we can not apply Theorem \ref{thmA} because of the lacking of the weak amenability. In this section, we study some partial answers to this problem. 

	Observe that $L(\Z^2\rtimes \mathrm{SL}(2,\Z))$ has two structures: one is the crossed product $L^\infty(\mathbb{T}^2)\rtimes \mathrm{SL}(2,\Z)$ coming from a strongly ergodic action of a bi-exact weakly amenable group; and the other is a bi-exact group factor \cite{Oz08}. 
From these viewpoints, we give partial answers to the property (TFF) as follows. 
See \cite[Definition 12.3.9]{BO08} for the definition of the W$^*$CMAP (or equivalently, the W$^*$CBAP with Cowling--Haagerup constant 1).

\begin{Pro}\label{partial results}
	The following statements hold true.
\begin{itemize}
	\item[$(1)$] Let $\Gamma$ be a non-amenable, weakly amenable, and bi-exact group acting on a standard probability space $X$ as a free, strongly ergodic, and p.m.p.\ action. Put $M:= L^\infty(X)\rtimes \Gamma$. Then for any full $\rm II_1$ factor $B$, one has $\mathcal{F}(B\ovt M)=\mathcal{F}(B)\mathcal{F}(M)$.
	\item[$(2)$] Let $\Gamma$ be a non-amenable bi-exact ICC group. Then for any $\rm II_1$ factor $B$ with the W$^*$CMAP, one has $\mathcal{F}(B\ovt L\Gamma)=\mathcal{F}(B)\mathcal{F}(L\Gamma)$.
\end{itemize}
\end{Pro}

The first assertion of this proposition will be proved by combining the proof of \cite[Theorem 5.1.1]{Is14} with the following lemma.

\begin{Lem}[{\cite[Proposition 6.3]{Ho15}}]\label{lemma for full}
	Let $N=M\ovt B = K\ovt L$ be a tensor decomposition as $\rm II_1$ factors, and let $A \subset M$ be a Cartan subalgebra. If $K \preceq_N A\ovt B$ and $K$ is full, then we have $K \preceq_N B$.
\end{Lem}

\begin{proof}[Proof of Proposition \ref{partial results}(1)]
	We show that for any tensor decomposition $M\ovt B =  K\ovt L$ with $B$ full, one has $K \preceq_{M\ovt B} B$ or $L \preceq_{M\ovt B} B$. This gives the conclusion by Lemma \ref{tensor formula}.

	Observe that $M$ is full since the action is strongly ergodic. So by \cite[Corollary 2.3]{Co75}, the tensor product $M \ovt B$ is full, and hence so are $K$ and $L$. By Theorem \ref{PV12} and the proof of \cite[Theorem 5.1.1]{Is14}, one has $K \preceq_{M\ovt B} B\ovt L^\infty(X)$ or $L \preceq_{M\ovt B} B\ovt L^\infty(X)$. Then we can apply Lemma \ref{lemma for full}, and obtain $K \preceq_{M\ovt B} B$ or $L \preceq_{M\ovt B} B$. 
\end{proof}

	For the second assertion of Proposition \ref{partial results}, we prove the following proposition. This should be regarded as a ``relativization'' of Ozawa's semisolidity theorem \cite[Theorem 4.6]{Oz04}. Actually we can not give a complete generalization of Ozawa's theorem, since local reflexivity (or exactness) of $C^*_\lambda(\Gamma)$ is not enough as an extension property in this setting. We will use the W$^*$CMAP on $B$ to avoid this problem.

\begin{Pro}\label{AO with weakly amenable}
	Let $\Gamma$ be a bi-exact group and $B$ a finite von Neumann algebra. Put $M:=L\Gamma \ovt B$. Then for any von Neumann subalgebra $A \subset M$ with $A\not\preceq_MB$, there is a u.c.p.\ map from $\langle M, B \rangle$ into $A'\cap M$, which restricts to the conditional expectation $E_{A'\cap M}$ on $L\Gamma \otimes_{\rm min} B$. 

If $B$ has the W$^*$CMAP, then the resulting u.c.p.\ map can be taken as the one restricting $E_{A'\cap M}$ on $ L\Gamma \ovt B$, and thus $A'\cap M$ is amenable relative to $B$.
\end{Pro}
\begin{proof}
	Since most parts of the proof are straightforward ``relativization'' of the one of \cite[Theorem 4.6]{Oz04}, we give only a sketch. Our proof here is very similar to the one of \cite[Theorem 15.1.5]{BO08} (and its generalization \cite[Theorem 5.3.3]{Is12}).  
We will use the Hilbert space $H:=L^2(M)\otimes_B L^2(M) = \ell^2(\Gamma)\otimes L^2(B) \otimes \ell^2(\Gamma)$, in stead of $L^2(M)\otimes L^2(M)$. 

	The first part is exactly the same as the one of \cite[Theorem 15.1.5]{BO08} (and \cite[Theorem 5.3.3]{Is12}). Assume $A\not\preceq_{M}B$. By \cite[Corollary F.14]{BO08}, we may assume $A$ is abelian. Then one can define a proper conditional expectation 
\begin{equation*}
	\Psi_{A} \colon \mathbb{B}(L^2(M)) \longrightarrow A'\cap \mathbb{B}(L^2(M)).
\end{equation*}
The condition $A\not\preceq_{M}B$ implies $\Psi_A(\K(\ell^2(\Gamma))\otimes_{\rm min}\B(L^2(B)))=0$. 

	From now on, we use the relative tensor product. In particular, we will not use \cite[Proposition 4.2]{Oz04} but use a characterization of bi-exactness \cite[Lemma 15.1.4]{BO08}. 
Let $\pi_H$ and $\theta_H$ be left and right actions of $M$ on $H$, and denote by $\nu$ the algebraic $\ast$-homomorphism from $\pi_H(M)\theta_H(M^{\rm op})$ to $\B(L^2(M))$ given by $\nu(\pi_H(a)\theta_H(b^{\rm op}))=ab^{\rm op}$. 
Let $\Theta\colon C^*_\lambda(\Gamma)\otimes_{\rm min} C^*_\lambda(\Gamma)^{\rm op} \to \B(\ell^2(\Gamma))$ be a u.c.p.\ map such that $\Theta(a\otimes b^{\rm op}) - ab^{\rm op}\in \K(\ell^2(\Gamma))$ \cite[Lemma 15.1.4]{BO08}. 
Put $M_0:=C^*_\lambda(\Gamma)\otimes_{\rm min} B$. 
Identifying C$^* \{\pi_H(M_0), \theta_H(M_0) \}$ as $C^*_\lambda(\Gamma)\otimes_{\rm min} \mathrm{C}^*\{B,B^{\rm op}\} \otimes_{\rm min}C^*_\lambda(\Gamma)^{\rm op}$, we may define $\Theta$ on this algebra, which is the identity on $B$ and $B^{\rm op}$. 
Observe that at the C$^*$-algebra level, $\Theta$ and $\nu$ coincide \textit{modulo $\K(\ell^2(\Gamma))\otimes_{\rm min}\B(L^2(B))$}, that is, 
	$$\Theta(\pi_H(a)\theta_H(b^{\rm op})) - a b^{\rm op} \in \K(\ell^2(\Gamma))\otimes_{\rm min}\B(L^2(B)), \quad a,b \in M_0 .$$
Thus on C$^* \{ \pi_H(M_0), \theta_H(M_0)\}$ the composition $\Phi_A\circ \nu$  coincides with $\Phi_A\circ \Theta$, and hence is a bounded u.c.p.\ map.

	Observe that $\Phi_A|_M$ is the unique trace preserving conditional expectation $E_{A'\cap M} \colon M\to A'\cap M$, and hence in particular \textit{normal} on $M$. So the map $\Phi_A \circ \nu$ is a normal u.c.p.\ map on $\pi_H(M)$. Regarding again C$^* \{\pi_H(M_0), \theta_H(M_0) \} = C^*_\lambda(\Gamma)\otimes_{\rm min} \mathrm{C}^*\{B,B^{\rm op}\} \otimes_{\rm min}C^*_\lambda(\Gamma)^{\rm op}$, we can apply the local reflexivity of $C^*_\lambda(\Gamma)$ (this comes from exactness of $\Gamma$) and extend $\Phi_A\circ \nu$ on $L\Gamma \otimes_{\rm min} \mathrm{C}^*\{B,B^{\rm op}\} \otimes_{\rm min}C^*_\lambda(\Gamma)^{\rm op}$ which is normal on $L\Gamma$ (see Lemma 9.4.1, Proposition 9.2.5, and the proof of Lemma 9.2.9 in \cite{BO08} for these facts). 
Finally by Arveson's extension theorem, we again extend $\Phi_A \circ \nu$ on C$^*\{\pi_H(\langle M,B \rangle), \theta_H(M_0)\}$. Then the restriction on $\pi_H(\langle M,B \rangle)$ of the resulting map defines a u.c.p.\ map from $\langle M,B \rangle$ into $A'\cap (M_0^{\rm op})' =A'\cap M$. By construction, this is a desired item.

	Finally assume that $B$ has the W$^*$CMAP, and take a net $(\psi_i)_i$ of normal finite rank c.c.\ maps on $B$ converging to $\id_B$ point weakly. We extend these maps to $\langle M,B \rangle = \B(\ell^2(\Gamma))\ovt B$ by $\id \otimes \psi_i =: \widetilde{\psi}_i$. Observe that $\widetilde{\psi}_i (M) \subset L\Gamma \otimes_{\rm min} B$ for all $i$. Let $\Phi$ be the u.c.p.\ map constructed in the first half of the proof. If we take a cluster point $\widetilde{\Phi}$ of $(\Phi\circ \widetilde{\psi}_i)_i$, then this is a c.c.\ map from $\langle M,B \rangle$ into $A'\cap  M$ which restricts to $E_{A'\cap M}$ on $M$. In fact, for any $x \in M \ovt B$, one has 
	$$\Phi\circ \widetilde{\psi}_i (x) = E_{A'\cap M}\circ \widetilde{\psi}_i (x) \to E_{A'\cap M}(x), \quad \text{as } i\to \infty.$$
Hence $\widetilde{\Phi}|_M = E_{A'\cap M}$ and $\widetilde{\Phi}$ is a conditional expectation onto $A'\cap M$. 
\end{proof}

\begin{proof}[Proof of Proposition \ref{partial results}(2)]
	Take $t \in \mathcal{F}(L\Gamma \ovt B)$ and fix $M:=L\Gamma \ovt B = L\Gamma \ovt B^t (=: K \ovt L)$. 
By (the proof of) Lemma \ref{tensor formula}, we have only to show that 
$K\preceq_{M} B$, $L\preceq_{M} B$, $L\Gamma \preceq_{M}L$, or  $B\preceq_{M}L$.
So suppose by contradiction that each of them does not happen.

	We apply Proposition \ref{AO with weakly amenable} to $K$ (actually an abelian subalgebra of $K$ by \cite[Corollary F.14]{BO08}), and get that $L\lessdot_M B$. By exchanging the roles, we also have that $L\Gamma \lessdot_M L$ and hence $L\Gamma \lessdot_M B$ by Proposition \ref{relative amenable transitivity}. Thus by Lemma \ref{relative amenable and amenable}, we obtain amenability of $L\Gamma$ which is a contradiction.
\end{proof}

\small{

}

\begin{thebibliography}{CKP14}
%
%
\bibitem[BO08]{BO08} N.~P.~Brown and N.~Ozawa, \textit{$C^*$-algebras and finite-dimensional approximations}. Graduate Studies in Mathematics, 88. American Mathematical Society, Providence, RI, 2008.
%
%
\bibitem[CKP14]{CKP14}  I. Chifan, Y. Kida and S. Pant, \textit{Primeness results for von Neumann algebras associated with surface braid groups.} Int. Math. Res. Not. (2015), article id:rnv271, 42pp.
%
\bibitem[CSU11]{CSU11} I. Chifan, T. Sinclair and B. Udrea, \textit{On the structural theory of $\rm II_1$ factors of negatively curved groups, $\rm II$. Actions by product groups.} Adv. Math. {\bf 245} (2013), 208--236.
%
%
\bibitem[Co75]{Co75} A.~Connes, \textit{Classification of injective factors. Cases $\rm II_1$, $\rm II_\infty$, $\rm III_\lambda$, $\lambda\neq 1$.} Ann.\ of Math.\  (2) {\bf 104} (1976), 73--115.
%
\bibitem[Co80]{Co80} A.~Connes, \textit{A factor of type $\rm II_1$ with countable fundamental group}. J. Operator Theory {\bf 4} (1980), 151--153.
%
\bibitem[De10]{De10} S.~Deprez, \textit{Explicit examples of equivalence relations and factors with prescribed fundamental group and outer automorphism group}. Preprint, \texttt{arXiv:1010.3612}.
%
%
\bibitem[DR99]{DR99} K. J. Dykema and F. R$\rm\breve{a}$dulescu, \textit{Compressions of free products of von Neumann algebras}. Math. Ann., {\bf 316} (2000), 61--82.
%
\bibitem[Ga99]{Ga99} D. Gaboriau, \textit{Co$\it \hat{u}$t des relations d'\'{e}quivalence et des groupes.} Invent. Math. {\bf 139} (2000), no. 1, 41--98.
%
\bibitem[Ga01]{Ga01} D. Gaboriau, \textit{Invariants $\ell^2$ de relations d'\'{e}quivalence et de groupes}. Publ. Math. Inst. Hautes \'{E}tudes Sci. No. {\bf 95} (2002), 93--150.
%
\bibitem[Ho15]{Ho15}  D.J. Hoff, \textit{Von Neumann algebras of equivalence relations with nontrivial one-cohomology.} J. Funct. Anal. {\bf 270} (2016), no. 4, 1501--1536.
%
\bibitem[Ho07]{Ho07} C.~Houdayer, \textit{Construction of type $\rm II_1$ factors with prescribed countable fundamental group}. J. Reine Angew Math. {\bf 634} (2009), 169--207.
%
%
%
\bibitem[HI15]{HI15} C.~Houdayer and Y~.Isono, \textit{Unique prime factorization and bicentralizer problem for a class of type $\rm III$ factors}. Adv. Math. {\bf 305} (2017), 402--455.
%
\bibitem[HU15]{HU15} C. Houdayer and Y. Ueda, \textit{Rigidity of free product von Neumann algebras.} Compos. Math. {\bf 152} (2016), 2461--2492.
%
\bibitem[Io12]{Io12} A.~Ioana, \textit{Cartan subalgebras of amalgamated free product $\rm II_1$ factors} (with an appendix joint with S.~Vaes). Ann.\ Sci.\ \'{E}cole Norm.\ Sup. {\bf 48} (2015), 71--130.
%
\bibitem[IPP05]{IPP05} A.~Ioana, J.~Peterson and S.~Popa, \textit{Amalgamated free products of weakly rigid factors and calculation of their symmetry groups}. Acta Math.\ {\bf 200} (2008), 85--153.
%
\bibitem[Is12]{Is12} Y.~Isono, \textit{Weak Exactness for C$^*$-algebras and Application to Condition (AO)}, J.\ Funct.\ Anal.\ {\bf 264} (2013), 964--998.
%
%
%
\bibitem[Is14]{Is14} Y.~Isono, \textit{Some prime factorization results for free quantum group factors}. J. Reine Angew. Math. {\bf 722} (2017), 215--250.
%
\bibitem[MV43]{MV43} F.J.~Murray and J.~Von Neumann, \textit{On rings of operators} IV. Ann.\ Math.\ {\bf 44} (1943), 716--808.
%
\bibitem[Oz04]{Oz04} N.~Ozawa, \textit{A Kurosh type theorem for type $\rm II_1$ factors}. Int.\ Math.\ Res.\ Not.\ (2006), Art. ID 97560, 21 pp.
%
\bibitem[Oz08]{Oz08} N.~Ozawa, \textit{An example of a solid von Neumann algebra.} Hokkaido Math. J., {\bf 38} (2009), 557--561.
%
%
%
\bibitem[OP03]{OP03} N.~Ozawa and S.~Popa, \textit{Some prime factorization results for type $\rm II_1$ factors}. Invent.\ Math.\ {\bf 156} (2004), 223--234.  
%
\bibitem[OP07]{OP07} N.~Ozawa and S.~Popa, \textit{On a class of $\rm II_1$ factors with at most one Cartan subalgebra}. Ann.\ of Math.\ (2), {\bf 172} (2010), 713--749.
%
%
\bibitem[Pe06]{Pe06} J.~Peterson, \textit{L$^2$-rigidity in von Neumann algebras}, Invent.\ Math.\ {\bf 175} (2009), 417--433.
%
\bibitem[Po01]{Po01} S.~Popa, \textit{On a class of type $\rm II_1$ factors with Betti numbers invariants}. Ann.\ of Math.\ {\bf 163} (2006), 809--899.
%
\bibitem[Po03]{Po03} S.~Popa, \textit{Strong rigidity of $\rm II_1$ factors arising from malleable actions of w-rigid groups $\rm I$}. Invent.\ Math.\ {\bf 165} (2006), 369--408.
%
\bibitem[Po04]{Po04} S.~Popa, \textit{Strong rigidity of $\rm II_1$ factors arising from malleable actions of w-rigid groups}, II. Invent. Math. {\bf 165} (2006), 409--452.
%
\bibitem[Po06a]{Po06a} S.~Popa, \textit{On the superrigidity of malleable actions with spectral gap}. J. Amer. Math. Soc. {\bf 21} (2008), 981--1000.
%
\bibitem[Po06b]{Po06b} S.~Popa, \textit{On Ozawa's property for free group factors}. Int. Math. Res. Notices. Vol. 2007 : article ID rnm036, 10 pages.
%
\bibitem[PV08a]{PV08a} S.~Popa and S.~Vaes, \textit{Actions of $\F_\infty$ whose $\rm II_1$ factors and orbit equivalence relations have prescribed fundamental group}. J. Amer. Math. Soc. {\bf 23} (2010), 383--403.
%
\bibitem[PV08b]{PV08b} S.~Popa and S.~Vaes, \textit{On the fundamental group of $\rm II_1$ factors and equivalence relations arising from group actions}. In Quanta of Maths, Proceedings of the Conference in honor of A. Connes' 60th birthday, Clay Mathematics Institute Proceedings, {\bf 11} (2011), pp. 519--541.
%
%
\bibitem[PV12]{PV12} S.~Popa and S.~Vaes, \textit{Unique Cartan decomposition for $\rm II_1$ factors arising from arbitrary actions of hyperbolic groups}. J.\ Reine Angew.\ Math. {\bf 694} (2014), 215--239.
%
\bibitem[Ra91]{Ra91} F.~R$\rm\breve{a}$dulescu, \textit{The fundamental group of the von Neumann algebra of a free group with infinitely many generators is $\R_+^*$}. J. Amer. Math. Soc. {\bf 5} (1992), 517--532.
%
\bibitem[Sa09]{Sa09} H.~Sako, \textit{Measure equivalence rigidity and bi-exactness of groups}. J.\ Funct.\ Anal.\ {\bf 257} (2009) 3167--3202.
%
\bibitem[SW11]{SW11} J.~O.~Sizemore and A.~Winchester, \textit{A unique prime decomposition result for wreath product factors}. Pacific J.\ Math.\ {\bf 265} (2013), no. 1, 221--232.
%
%
%
%
%
%
\bibitem[Va08]{Va08} S.~Vaes, \textit{Explicit computations of all finite index bimodules for a family of $\rm II_1$ factors}. Ann.\ Sci.\ \'{E}cole Norm.\ Sup.\ {\bf 41} (2008), 743--788.
%
%
\bibitem[Va13]{Va13} S.~Vaes, \textit{Normalizers inside amalgamated free product von Neumann algebras}. Publ.\ Res.\ Inst.\ Math.\ Sci.\ {\bf 50} (2014), 695--721.
%
\bibitem[Vo89]{Vo89} D.V.~Voiculescu, \textit{Circular and semicircular systems and free product factors}. In Operator algebras, unitary representations, enveloping algebras, and invariant theory (Paris, 1989), Progr. Math. {\bf 92}, Birkh$\rm \ddot{a}$user, Boston, 1990, p. 45--60.
%
\end{thebibliography}
\end{document}